\documentclass[12pt]{amsart}
\usepackage[
  hmarginratio={1:1},     
  vmarginratio={1:1},     
  textwidth=148mm,        
  textheight=220mm,       
  marginparwidth=21mm,
  marginparsep=3mm
]{geometry}

\usepackage[utf8]{inputenc}
\usepackage{lmodern}
\usepackage{url}
\usepackage{amsmath, amsthm, amssymb, amscd, accents, bm}
\usepackage{mathtools}
\usepackage{mdwlist}
\usepackage{paralist}
\usepackage{graphicx}
\usepackage{epstopdf}
\usepackage{datetime}
\usepackage{afterpage}
\usepackage[dvipsnames]{xcolor}
\usepackage{hyperref}
\usepackage[ruled,vlined]{algorithm2e}
\usepackage{caption}
\usepackage{subcaption}
\mathtoolsset{showonlyrefs}
\usepackage[sort,nocompress]{cite}
\usepackage{marginnote}
\usepackage{comment}
\usepackage{bm}
\usepackage{scalerel}
\usepackage[normalem]{ulem}

\usepackage{mathtools}
\usepackage{nicematrix}

\mathtoolsset{showonlyrefs}

\theoremstyle{plain}
\newtheorem{definition}{Definition}[section]
\newtheorem{theorem}[definition]{Theorem}
\newtheorem{lemma}[definition]{Lemma}
\newtheorem{corollary}[definition]{Corollary}
\newtheorem{proposition}[definition]{Proposition}

\newtheorem*{theorem*}{Theorem}

\theoremstyle{definition}
\newtheorem{remark}[definition]{Remark}
\newtheorem*{remark*}{Remark}

\newtheorem*{remarks*}{Remarks}

\numberwithin{equation}{section}

\def\N{{\mathbb N}}
\def\Z{{\mathbb Z}}
\def\R{{\mathbb R}}
\def\T{{\mathbb T}}
\def\C{{\mathbb C}}
\def\Q{{\mathbb Q}}
\newcommand{\Rd}{{\R^d}}
\newcommand{\Rdd}{{\R^{2d}}}
\newcommand{\Rddd}{{\R^{4d}}}

\newcommand{\Sd}{{\mathcal{S}(\Rd)}}
\newcommand{\Sdd}{{\mathcal{S}(\Rdd)}}

\newcommand{\Sdp}{{\mathcal{S}'(\Rd)}}
\newcommand{\Sddp}{{\mathcal{S}'(\Rdd)}}

\newcommand{\ltd}{{L^2(\R^d)}}
\newcommand{\ltdd}{{L^2(\Rdd)}}

\newcommand{\rep}{{\mathfrak{R}}}
\newcommand{\Trep}{{\mathfrak{T}}}

\newcommand{\J}{{\mathcal{J}}}
\newcommand{\A}{\mathcal{A}}

\makeatletter
\newcommand\thankssymb[1]{\textsuperscript{\@fnsymbol{#1}}}

\newcommand{\Sp}{{\mathrm{Sp}(4d,\R)}}
\newcommand{\Spp}{{\mathrm{Sp}(2d,\R)}}

\makeatletter
\def\@makefnmark{%
  \leavevmode
  \raise.9ex\hbox{\fontsize\sf@size\z@\normalfont\tiny\@thefnmark}}

\makeatletter
\def\bign#1{\mathclose{\hbox{$\left#1\vbox to8.5\p@{}\right.\n@space$}}\mathopen{}}

\usepackage{xifthen}
\usepackage{eqparbox}
\usepackage{makebox}

\newcommand{\norm}[1]{\lVert #1\rVert}



\newcommand{\tfa}{time-frequency analysis}
\newcommand{\tfr}{time-frequency representation}
\newcommand{\tfs}{time-frequency shift}
\newcommand{\ft}{Fourier transform}
\newcommand{\stft}{short-time Fourier transform}

\newcommand{\tf}{time-frequency}
\newcommand{\psdo}{pseudodifferential operator}

\newcommand{\bR}{\R}        
\newcommand{\bN}{\N}        
\newcommand{\bQ}{\Q}        
\def\rd{\bR^d}
\def\rdd{{\bR^{2d}}}
\def\lrd{L^2(\rd)}

\def\inv{^{-1}}

\usepackage{amsmath,amsfonts,amsthm,amsopn,cite,mathrsfs}
\usepackage{epsfig,verbatim}
\begin{document}
\begin{abstract}
We characterize all time-frequency representations that satisfy a
general covariance property: 
any weak*-continuous bilinear mapping that intertwines
time-frequency shifts on the configuration space with time-frequency shifts on phase space is a multiple of a metaplectic time-frequency representation. 
\end{abstract}

\title[A characterization of metaplectic time-frequency representations]{A Characterization of Metaplectic Time-Frequency Representations}

\author{Karlheinz Gr\"ochenig}
\address{Faculty of Mathematics \\
University of Vienna \\
Oskar-Morgenstern-Platz 1 \\
A-1090 Vienna, Austria}
\email{karlheinz.groechenig@univie.ac.at}
\author{Irina Shafkulovska}
\address{Faculty of Mathematics \\
University of Vienna \\
Oskar-Morgenstern-Platz 1 \\
A-1090 Vienna, Austria}
\email{irina.shafkulovska@univie.ac.at}

\thanks{I.~Shafkulovska was funded in part by the Austrian Science Fund (FWF) [\href{https://doi.org/10.55776/P33217}{10.55776/P33217}] and [\href{https://doi.org/10.55776/Y1199}{10.55776/Y1199}]. 
 For open access purposes, the authors have applied a CC BY public copyright license to any author-accepted manuscript version arising from this submission.}

\maketitle

\section{Introduction} 

Many important \tfr s possess the covariance property: a \tfs\ of a
function is transformed to a shift of the \tfr . For one of the most
central \tfr s, namely the Wigner distribution, this property looks as
follows.  Let 
  \begin{equation}
     W(f,g)(x,\omega) =  \int_\Rd f(x+\tfrac{t}{2})
     \overline{g(x-\tfrac{t}{2})}e^{-2\pi i \omega\cdot t}\, dt,\quad
     x,\omega\in\Rd \, ,
 \end{equation}
be the (cross-) Wigner distribution of two functions $f,g \in \lrd
$. The translation operator or shift on $\lrd $ is defined  by $T_xf(t) =
f(t-x)$ for  $t,x,\omega \in \rd $, and  
 the frequency shift or modulation operator is given by  $M_\omega f(t) = e^{2\pi i \omega \cdot
  t} f(t) $. Its \ft\ is $\widehat{M_\omega f} = T_{-\omega }\hat{f}$, whence
frequency shift. Writing $\lambda = (x,\omega )$ for a point in the
\tf\ space, the  symmetrized composition is the \tfs\
$$
\rho(\lambda) =\rho(x,\omega) = T_{x/2} M_\omega T_{x/2},\quad \lambda
= (x,\omega)\in\Rdd\, .
$$ 
Then the covariance property states that for all $\lambda, z \in \rdd,
f,g\in \lrd $
\begin{equation}
  \label{eq:n4}
W(\rho (\lambda )f, \rho (\lambda )g)(z) = W(f,g)(z-\lambda ) =
(T_\lambda W(f,g))(z) \, .  
\end{equation}
Many other useful  \tfr s possess this useful and intuitive property, for instance, the
spectrogram or the Born-Jordan distribution.
Therefore, covariant \tfr s have been studied
intensely in the engineering literature and in physics, and ultimately
were completely classified by L.\ Cohen~\cite{cohen66,cohen}.  If a
\tfr\ $\rep (f,g)$  is covariant, then $\rep $ must have the form
$\rep (f,g) = W(f,g) \ast \Psi $ for some distribution $\Psi \in \Sddp$. See also \cite[Sec.~4.5]{Groechenig2001} for a mathematical treatment of
Cohen's class.

A second look at the Wigner distribution reveals a more delicate structure. 
By viewing the Wigner distribution as a sequilinear mapping, we can apply
different \tfs s to $f$ and $g$ and obtain the following
formula~(e.g., \cite[Prop.~1.94]{Folland1989} or \cite[Prop.~4.3.2]{Groechenig2001})
\begin{align} 
  W(\rho (\lambda )f, \rho (\mu )g)(z) &=  e^{\pi i \lambda \cdot \J \mu
    } e^{2\pi i z \cdot
    \J(\lambda - \mu )}  W(f,g)\big(z- \frac{\lambda
    +\mu}{2}\big) \notag \\
  &= \rho \Big( \tfrac{\lambda +\mu }{2}, \J (\lambda -\mu) \Big) W(f,g)(z) \, ,   \label{eq:n2}
\end{align}
with $\J = \begin{psmallmatrix}
    0 & I_d \\ -I_d & 0
           \end{psmallmatrix}$.
Thus the Wigner distribution with two \emph{different} \tfs s is  not
only 
a shift  of the  Wigner distribution, but also contains a modulation, in other words, the result is a \tfs\ of
    the original Wigner distribution. Formula~\eqref{eq:n2} is a general
    covariance property, and it says that the transform of \tfs s is
    again a \tfs .  Note, however, that on the
    right-hand side the \tfs\ has twice as many variables. The general
    covariance property \eqref{eq:n2} is a fundamental property of the Wigner
    distribution and plays an important role  in the \tf\ approach
    to \psdo s, see e.g., ~\cite{gro06} and the recent
    monograph~\cite{CR20}.

    We say that a \tfr  ,  i.e., a bilinear mapping $\rep : \Sd \times
    \Sd \to \Sddp $, possesses the general covariance property, if it
    maps \tfs s of $f,g\in \Sd$ to a \tfs\ of $\rep (f,g)$. 
    Formally,
    there exists a function $\Phi _{\rep} : \bR ^{4d } \to  \bR ^{4d }
    $, such that for all $f,g\in \Sd$ and all $\lambda , \mu \in \rdd
    $ one has  
\begin{equation} \label{eq:vertauschung}
  \rep(\rho(\lambda)f,\rho(\gamma)g) = c
  \rho(\Phi_\rep(\lambda,\gamma)) \rep(f,g). 
    \end{equation}
 Note that in this definition  $\rho $ occurs with different dimensions. 
 On the left-hand side 
 $\rho (\lambda
 )f$  acts on a function of $d$ variables, on the right hand side
 $\rho (\Phi _{\rep} (\lambda , \mu ))$ acts on a distribution on  $\bR^{2d} $. 
 In the sequel we will use only the latter version, therefore, we will not introduce a different notation.

It is natural to assume that $\Phi$ is measurable. 
In  this paper, we derive a complete classification of all \tfr s  that satisfy the
general covariance property~\eqref{eq:vertauschung}.   

For the formulation of the main result, we recall the definition of a
metaplectic operator. By the theorem of Stone-von Neumann~\cite{Folland1989} for every
symplectic matrix $\A \in \Spp $ there  exists a unitary
operator $\hat {\A }$ on $\lrd $, such that
\begin{equation}
    \rho(\A\lambda) = \hat \A\rho(\lambda)\hat\A^{-1}, 
\end{equation}
that is, $\hat\A$ arises as an intertwining operator, and is called
the \emph{metaplectic operator} associated to $\A$.   

With these preparations, the main result about \tfr s obeying the general covariance property goes as follows.

\begin{theorem}[Main result]\label{thm:main} 
    Let $\rep:\Sd\times \Sd\to \Sddp$ be a non-zero, bilinear,
    separately weak*-continuous mapping satisfying the general
    covariance property ~\eqref{eq:vertauschung} with a measurable function $\Phi$. 
    Then there exist $a\in\C$, $a\neq 0$, and $\A\in\Sp$ such that   
    \begin{equation} 
        \rep(f,g) = a\, \hat\A (f\otimes g),\qquad f,g\in\Sd.
    \end{equation}
  \end{theorem}

The general covariance property seems to be very strong. Although we
start with a mapping from $\Sd \times \Sd \to \Sddp$, the actual
mapping properties are much stronger. 
  
\begin{corollary}\label{cor:from_main}
    Let $\rep:\Sd\times \Sd\to \Sddp$. Under the assumptions of
    Theorem~\ref{thm:main},  $\rep$ has the following properties. 
    \begin{enumerate}[(a)] 
    \item The function $\Phi_\rep$ is unique and linear.
    \item   $\rep$ maps  $\Sd\times \Sd$ into  $\Sdd$.
    \item $\rep$ is non-degenerate. 
    \item $\rep$ extends to an  isometry $\rep:\ltd\times\ltd\to\ltdd$, i.e., 
    \begin{equation}
        \norm{\rep (f,g)}_2 = |a|\,\norm{f}_2\norm{g}_2,\quad f,g\in\ltd. 
      \end{equation}
    \item The range $\{ \rep (f,g): f,g\in \Sd \}$ spans a dense
      subspace of $\Sdd $, of $\lrd $, and a weak$^*$-dense subspace
      of $\Sddp$. 
    \end{enumerate}
\end{corollary}

Note that we only  assume that $\Phi _\rep $ is
measurable! 

It turns out that \tfr s of the form   $\hat{\A} (f\otimes \bar{g})$
were recently introduced by Cordero and
Rodino~\cite{CorderoEtAl2022} as a generalization of the Wigner
distribution, the \stft , 
or Beyer's \tfr s~\cite{bayer10}. They are called  metaplectic \tfr s or
the metaplectic Wigner distributions associated to $\A$, or shortly,
the $\A$-Wigner distributions.  
In a series of papers, Cordero, Rodino, and their
collaborators ~\cite{CorderoRodino2023,CorderoEtAl2024,CorderoRodino2022} have shown that metaplectic \tfr s are a versatile tool for the refined investigation of \psdo s and function spaces. 
Uncertainty principles for metaplectic \tfr s were investigated
in~\cite{GroechenigShafkulovska2024,GroechenigShafkulovska2025}.

Theorem~\ref{thm:main} provides an intrinsic, structural argument as
to why metaplectic \tfr s are natural and are bound to appear in \tfa. 

\begin{remark}
Usually, a \tfr\ is treated as a sesquilinear form rather than 
as a bilinear form.  
 Theorem \ref{thm:main} can be equivalently restated in terms of
 sesquilinear forms by a tiny modification. If $\rep $ is bilinear,
 then 
 the corresponding sesquilinear form  $\overline{\rep}(f,g)\coloneqq \rep(f,\overline{g})$
is non-zero, continuous, and satisfies
\begin{equation*}
    \overline{\rep}(\rho(\lambda)f,\rho(\gamma)g) = c
    \rho(\Phi_{\overline \rep}(\lambda,\gamma)) \overline{\rep}(f,g), 
    \quad 
    \Phi_{\overline{\rep}}(\lambda_1,\lambda_2,\gamma_1,\gamma_2) =
    \Phi_{{\rep}}(\lambda_1,\lambda_2,\gamma_1,-\gamma_2). 
\end{equation*}
The change of the intertwining function is due to
$\overline{\rho(\gamma_1,\gamma_2) g} = \rho(\gamma_1,-\gamma_2)
\overline{g}$. The conclusion of Theorem~\ref{thm:main}  $\overline{\rep} (f,g) = a \, \hat{\A }
(f \otimes \bar{g})$ remains the same. 
\end{remark}

The remainder of the paper is devoted to the proof of Theorem~\ref{thm:main}. 
We first use the continuity assumptions of $\rep$ to identify it with a continuous linear mapping $\Trep:\Sdd\to\Sddp$ that satisfies a general covariance property. 
Here, the arguments are analytical and are based on a bilinear
version of the Schwartz kernel theorem.
In the next step, we exploit the underlying Lie group structure of the problem to  
deduce that $\Phi $ must be a linear
mapping.
Finally, 
the general covariance
property~\eqref{eq:vertauschung} 
implies that $\Phi$ is symplectic. The arguments here are purely algebraic. 
We put these three aspects together to derive the final form of
$\rep$.

We would like to offer our  thanks and appreciation  to Christian Bargetz whose
advice helped us to compress our original proof of Lemma~\ref{cor:rep_kernel} to a few lines.

\section{Symplectic matrices and metaplectic operators}\label{sec:prelim}
The symmetric time-frequency shifts 
$$
\rho(\lambda) =\rho(x,\omega) = T_{x/2} M_\omega T_{x/2},\quad \lambda = (x,\omega)\in\Rdd,
$$ 
are a projective unitary representation of $\rdd$.
They satisfy
\begin{equation}\label{eq:rho_of_sum}
    \rho(\lambda_1+\lambda_2) =e^{\pi i \lambda_1\cdot \J \lambda_2} \rho(\lambda_1)\rho(\lambda_2),\qquad \J = \begin{psmallmatrix}
    0 & I_d \\ -I_d & 0
\end{psmallmatrix}.
\end{equation}
In particular, for all $n\in\N$ and $\lambda\in\Rd$ holds 
\begin{equation}\label{eq:rho_of_sum_special}
    \rho(\lambda)^{-1} = \rho(-\lambda),\qquad \text{ and } \quad \rho(n\lambda) = \rho(\lambda)^n 
\end{equation} 
The symplectic matrices are those matrices which preserve the standard symplectic form $[\lambda_1, \lambda_2] =\lambda_1\cdot \J\lambda_2$, i.e.,  
\begin{equation}\label{eq:Spp_def}
\Spp = \left\lbrace \A\in \R^{2d\times 2d}: \A^t \J \A = \J \right\rbrace.
\end{equation}
 This is a distinguished group of matrices in \tf \ analysis
 \cite{Groechenig2001, Folland1989, Gosson2011}. By the Stone-von
 Neumann theorem, for every symplectic matrix $\A\in\R^{2d\times 2d}$,
 there exists a  unitary operator $\A $ that is unique up to scaling
 with $c\in\T$  \cite{Folland1989, Groechenig2001, Gosson2011}
 satisfying 
\begin{equation}\label{eq:Stone_von-Neumann} 
    \rho(\A\lambda) = \hat \A\rho(\lambda)\hat\A^{-1}\, .
\end{equation}
Thus  $\hat\A$ arises as an intertwining operator, and is called the \emph{metaplectic operator} associated to $\A$.  
The metaplectic operators are isomorphisms on the Schwartz space $\Sd$, the space of rapidly decaying functions \cite{Folland1989, Gosson2011,Groechenig2001},
and by duality, on its dual space $\Sdp$.

\section{Analytic consequences} 
In this section, we exploit the weak*-continuity of the bilinear form $\rep$. 
Explicitly this means that for every sequence $(f_n)_{n\in\N}\in\Sd$ converging to $f\in\Sd$ in the topology of $\Sd$ and every $g\in\Sd$, we have
\begin{equation}
   \lim\limits_{n\to\infty} \rep(f_n, g) = \rep(f,g)\quad \text{ in the weak*-topology on }\Sdp,
\end{equation}
and, likewise, for all $(g_n)_{n\in\N}\in\Sd$ converging to $g\in\Sd$ and all  $f\in\Sd$ 
\begin{equation}
    \lim\limits_{n\to \infty}\rep(f, g_n)= \rep(f,g)\quad \text{ in the weak*-topology on }\Sdp.
\end{equation}
Throughout, we denote with $\langle T,f\rangle = T(f)$ the evaluation of the tempered distribution $T\in\Sdp$ at $f\in\Sd$. 
We will use the following bilinear version of  the Schwartz kernel theorem \cite[p.~531]{Treves1967}.
\begin{lemma}\label{cor:rep_kernel} 
    For every separately weak*-continuous bilinear mapping 
    $$
    \rep:\Sd\times \Sd\to\Sddp,
    $$
    there exists a unique weak*-continuous linear operator $\Trep:\Sdd\to \Sddp$ 
    such that for all $f,g\in\Sd$ and all $H\in\Sdd$
    \begin{equation}
        \langle \rep(f,g), H\rangle = \langle\Trep (f\otimes g), H\rangle.
    \end{equation}
    In addition, there exists a kernel $K=K_\rep\in \mathcal{S}'(\Rddd)$ such that 
    \begin{equation} \label{eq:kern1}
        \langle\rep(f,g),H\rangle = 
        \langle K, f\otimes g \otimes H \rangle,\quad f,g\in\Sd, \, H\in\Sdd.
    \end{equation}
\end{lemma}

\begin{proof}
As we are unaware of an explicit reference, we offer a proof sketch
for completeness. The proof uses the same arguments as the numerous
kernel theorems in functional analysis. 
We recall that $\Sd $ is a
nuclear Fr\'echet space and hence a Montel space, i.e., bounded,
closed sets in $\Sd $ are compact and vice versa, see~\cite[Prop.~34.4]{Treves1967}. This
property implies that a weak$^*$-continuous map from $\Sd $ to $\Sd '$
is automatically continuous, when $\Sd ' $ is endowed with the strong
topology (topology of uniform convergence on bounded
sets)~\cite[p.~358]{Treves1967}. Consequently, for fixed $g\in \Sd$ the map $f \in \Sd \mapsto
\Trep (f, g)\in \Sddp$ is strongly continuous, and likewise $g \in \Sd \mapsto
\Trep (f, g)\in \Sddp$  is strongly continuous for every $f \in \Sd
$. Since $\Sd $ is a Fr\'echet space, the separate continuity implies
the joint continuity of $\Trep $ on the Cartesian product $\Sd \times
\Sd $~\cite[Thm.~2.17]{rudin1986}.

Now we extend $\Trep $ to a linear map on the algebraic tensor product
$\Sd \times \Sd $. We endow $\Sd \times \Sd $ with the projective
tensor product topology. This is by definition the finest topology
such that the map $(f,g) \mapsto f\otimes g$ is continuous. 
As a consequence, $\Trep $ is (strongly) continuous on $\Sd \times \Sd
$. The completion of $\Sd \times \Sd$ with respect to this topology is
the projective tensor product $\Sd \otimes _\pi \, \Sd $ and is
canonically isomorphic with $\mathcal{S}(\rdd )$~\cite[Thm.~51.6.]{Treves1967}. By continuity,
$\Trep $  can now  be extended to a (strongly)  continuous map from $\Sd \otimes
_\pi \, \Sd \cong \Sdd $ to $\Sddp $. Now the Schwarz kernel theorem
provides a kernel proving the representation \eqref{eq:kern1}. 
\end{proof}

We now relate the bilinear kernel theorem to the setup in Theorem \ref{thm:main}.
We show that the linear extension $\Trep$ of $\rep$ also satisfies a covariance property.  
\begin{lemma}\label{lem:Trep_extension}
    Let $\rep:\Sd\times \Sd\to \Sddp$ be a non-zero, separately weak*-continuous bilinear form satisfying the covariance property \eqref{eq:vertauschung}. 
    Then there exists a non-zero weak*-continuous operator $\Trep : \Sdd\to\Sddp$ and  
    a function $\Phi:\Rddd\to\Rddd$ satisfying
    \begin{equation}\label{eq:Tdef}
    \Trep(f\otimes g) = \rep(f,g),\quad f,g\in\Sd
\end{equation}
and the covariance property
    \begin{equation}\label{eq:vertausch_Trep_extend}
    \Trep \rho(\lambda)F = c \rho(\Phi(\lambda)) \Trep F, \quad \lambda\in\Rddd,\, F\in\Sdd.
\end{equation}
The function $\Phi$ is measurable if and only if $\Phi_\rep$ is measurable. 
\end{lemma}
\begin{proof} 
Since $\rep $ is separately weak$^*$-continuous,  Lemma
\ref{cor:rep_kernel} asserts the existence of the associated linear operator
$\Trep $.
We claim that 
\begin{equation}\label{eq:defPhi}
    \Phi(\lambda)= \Phi(\lambda_1, \lambda_2, \lambda_3, \lambda_4) = \Phi_\rep(\lambda_1,  \lambda_3,\lambda_2, \lambda_4),\quad \lambda_1, \lambda_2, \lambda_3, \lambda_4\in\Rd.
\end{equation}
has the desired property.
The modification of $\Phi $ is a
    consequence of the reordering of \tfs s  for $\rho $ on $\rdd
    \times \rdd $ to $\R ^{4d}$ as follows: 
\begin{equation}\label{eq:rho_tensor}
    \rho(\lambda_1, \lambda_2, \lambda_3, \lambda_4) (f\otimes g) 
    = \rho(\lambda_1,\lambda_3)f\otimes \rho(\lambda_2,\lambda_4) g, 
\quad\lambda_1,\lambda_2,\lambda_3, \lambda_4\in\Rd.
\end{equation}
This implies that \eqref{eq:vertausch_Trep_extend}
is satisfied by the elementary tensors, and by linearity, on their linear span.
Since \tf\ shifts are weak*-continuous on $\Sddp$ \cite[Cor.~11.2.22]{Groechenig2001}, the identity $\Trep\rho(\lambda) = \rho(\Phi(\lambda))\Trep$ 
extends from $\Sd\otimes \Sd$ to $\Sdd$ by continuity and density. 
\end{proof}

\section{ Consequences of the General Covariance Property} 
In this section, we focus on the implications of the
covariance property \eqref{eq:vertauschung}, and by extension, of the covariance property \eqref{eq:vertausch_Trep_extend},
and show that  $\Phi$ is necessarily a symplectic linear
transformation.  For the proof we will use some basic, but deep facts
about homomorphisms between Lie groups. 

We split the proof into several lemmata.

\begin{lemma}\label{lem:injective} 
Let $\Trep:\Sdd\to\Sddp$ be a non-zero continuous linear operator satisfying the general covariance property \eqref{eq:vertausch_Trep_extend}. Then $\Trep$ is one-to-one.
\end{lemma}
\begin{proof}
If $\Trep F_0=0$ for a non-zero $F_0\in\Sdd$, then for all linear combinations $F = \sum_{j=1}^k c_j \rho(\lambda_j) F_0$ we obtain
\begin{equation}
    \Trep F = \sum_{j=1}^k c_j \Trep \rho(\lambda_j) F_0
= \sum_{j=1}^k c_j \rho(\Phi(\lambda_j)) \Trep F_0=0.
\end{equation}
Therefore, $\Trep $ vanishes on the dense subspace of finite linear combinations. By continuity, $\Trep =0$, contradicting the assumption.
\end{proof}
Lemma \ref{lem:injective} implies the no-degeneracy stated in Corollary \ref{cor:from_main} (c).

Recall that we denote the skew-symmetric form with $[\lambda,\mu] = \lambda\cdot \J\mu$, $\lambda,\mu\in\Rdd$.

\begin{lemma}\label{lem:Phi_of_Sum}
Let $\Trep:\Sdd\to\Sddp$ be a non-zero continuous linear operator satisfying the general covariance property \eqref{eq:vertausch_Trep_extend}. 
Then
\begin{equation}\label{eq:Phi_of_Sum}
    \rho(\Phi(\lambda+\mu))\Trep F = e^{\pi i ([\lambda,\mu] - [\Phi(\lambda),\Phi(\mu)])}\rho(\Phi(\lambda)+\Phi(\mu))\Trep F,\quad F\in\Sdd, \ \lambda,\mu\in\Rddd,
\end{equation}
and 
\begin{equation}\label{eq:Z_preserved}
[\lambda,\mu] - [\Phi(\lambda),\Phi(\mu)]\in \Z,\quad \lambda,\mu\in\Rddd.
\end{equation}
\end{lemma}
\begin{proof}
    Let $F\in\Sdd$ be a non-zero function, $\lambda,\mu\in\Rddd$. Then
    \begin{equation}
\begin{split}
    \rho(\Phi(\lambda+\mu))\Trep F 
    & = c^{-1}\,\Trep \rho(\lambda+\mu) F  
    = c^{-1}\, e^{\pi i [\lambda,\mu]}\Trep \rho(\lambda)\rho(\mu) F \\
    & = c\, e^{\pi i [\lambda,\mu]} \rho(\Phi(\lambda))\rho(\Phi(\mu)) \Trep F \\
    & = c\, e^{\pi i ([\lambda,\mu] - [\Phi(\lambda),\Phi(\mu)])} \rho(\Phi(\lambda)+\Phi(\mu)) \Trep F \\
    & = c\, e^{ -\pi i ([\lambda,\mu] - [\Phi(\lambda),\Phi(\mu)])} \rho(\Phi(\lambda)+\Phi(\mu)) \Trep F.
\end{split}
    \end{equation}
    The last equality here is obtained by switching the places of $\lambda$ and $\mu$ (recall that $[\lambda,\mu] = -[\mu,\lambda]$).
    By Lemma \ref{lem:injective}, the chain of equalities can only hold if the phase factors are equal, i.e., 
    \begin{equation}
        e^{2\pi i ([\lambda,\mu] - [\Phi(\lambda),\Phi(\mu)])} =1.
    \end{equation}
    This is equivalent to \eqref{eq:Z_preserved}.
\end{proof}
To show that $\Phi$ is linear, we introduce the subgroup 
\begin{equation}\label{eq:def_H}
    H = \lbrace
    \nu\in\Rddd : \ \exists c_\nu\in\C \text{ with }\rho(\nu)\Trep F = c_\nu\Trep F\text{ for all }F\in\Sdd
    \rbrace,
\end{equation}
This is the \emph{projective kernel} of the representation $\rho$ restricted to the subspace $\mathrm{ran}\,\Trep =\{ \Trep F\in\Sdd : F\in\Sdd\}$, and by continuity of $\Trep$ and $\rho(\lambda)$, $\lambda\in\Rddd$, 
to its closure $\overline{\mathrm{ran}\,\Trep}$ in $\Sdd $.
Let
\begin{equation}
    q_H:\Rddd\to \Rddd/H
\end{equation}
denote the quotient map.
\begin{lemma}\label{lem:H_properties}
Let $\Trep:\Sdd\to\Sddp$ be a non-zero continuous linear operator satisfying the general covariance property \eqref{eq:vertausch_Trep_extend}. 
Further, let $H$ be the projective kernel of $\rho$ defined in \eqref{eq:def_H}. 
Then $H$ is a closed subgroup of $\Rddd$.
\end{lemma}
\begin{proof}
    Since $\rho$ is a projective representation of $\Rddd$, it follows that $H$ is a group:
    \begin{equation}\label{eq:H_is_group}
        \rho(\nu+\eta)\Trep F = e^{\pi i [\nu,\eta]} c_\nu c_\eta \Trep F,\quad F\in\Sdd,\, \nu,\eta\in H.
    \end{equation}
    The continuity of the mapping
$        \lambda \mapsto \rho(\lambda)\Trep F$
    implies that $H$ is closed and the mapping $H\to\C,\ \nu\mapsto c_\nu,$
    is continuous. 
    To see this, let $(\nu_n)_{n\in\N}\subseteq H$ be a convergent sequence with limit $\mu\in\Rddd$. Then for all $F,G\in\Sdd$ 
    \begin{equation}
          \langle \rho(\nu) \Trep F,G\rangle
        = \langle \lim\limits_{n\to\infty}\rho(\nu_n) \Trep F,G\rangle
        = \lim\limits_{n\to\infty}  c_{\nu_n} \langle\Trep F,G\rangle.
    \end{equation}
    Since the limit in the middle converges, so does the limit on the right-hand side. Since $\Trep$ is not zero, there exist $F,G$ with $\langle \Trep F,G\rangle\neq 0$. Thus, 
    $ c_\nu = \lim\limits_{n\to\infty} c_{\nu_n}$
    is
    well-defined and satisfies $\rho(\nu)\Trep F = c_\nu\Trep F$, $F\in\Sdd$.
\end{proof}
\begin{lemma}\label{lem:Phi_almost_uniq}
Let $\Trep:\Sdd\to\Sddp$ be a non-zero continuous linear operator satisfying the general covariance property \eqref{eq:vertausch_Trep_extend}. 
Further, let $H$ be the projective kernel of $\rho$ defined in \eqref{eq:def_H}. 
Then 
the map
\begin{equation}
    q_H\circ \Phi:\Rddd\to \Rddd/H,\quad \lambda\mapsto (q_H\circ \Phi)(\lambda)
\end{equation}
is one-to-one.
\end{lemma}
\begin{proof}
We want to show that for all $\lambda,\mu\in\Rddd$ and $\nu\in H$ 
\begin{equation}
    \Phi(\lambda)=\Phi(\mu)+\nu\quad \implies \lambda = \mu.
\end{equation}
    In a single chain of inequalities, we have for all $F\in\Sdd$
\begin{align}
        \Trep (\rho(\lambda)F)
        & = c\,\rho(\Phi(\lambda))\Trep F = c\,\rho(\Phi(\mu)+\nu)\Trep F \\
        & \overset{\eqref{eq:rho_of_sum}}{=}c\, e^{\pi i [\Phi(\mu),\nu]}\rho(\Phi(\mu)) \rho(\nu) \Trep F \\
        &  = c\, c_\nu\, e^{\pi i [\Phi(\mu),\nu]}\rho(\Phi(\mu)) \Trep F \\
        &  = \Trep (c\, c_\nu\, e^{\pi i [\Phi(\mu),\nu]} \rho(\mu) F).
\end{align}
By Lemma \ref{lem:injective}, 
\begin{equation}
    \rho(\lambda)F = c\,c_\nu\, e^{\pi i [\Phi(\mu),\nu]} \rho(\mu) F,
\end{equation}
or equivalently, 
\begin{equation}\label{abc}
    \rho(\lambda-\mu)F =c\, c_\nu\, e^{\pi i ([\Phi(\mu),\nu] + [\lambda,\mu])} F.
\end{equation}
Since \eqref{abc}  holds  for all $F\in \Sdd$, the \tfs \ $\rho
(\lambda -\mu )$ must be a multiple of the identity, and therefore
$\lambda = \mu $. 
\end{proof} 

We now prove that $\Phi$ is an invertible linear map. 
\begin{proposition}\label{prop:Phi_linear}
    Let $\Trep:\Sdd\to\Sddp$ be a non-zero continuous linear operator satisfying the covariance property \eqref{eq:vertausch_Trep_extend} with a measurable $\Phi$. 
Then $\Phi$ is an invertible linear map, i.e., there exists a matrix $\A\in\mathrm{GL}(4d,\R)$ such that $\Phi(\lambda) = \A\lambda$, $\lambda\in\Rddd$.
\end{proposition}
\begin{proof}
Let $H$ be the projective kernel of $\rho$ defined in \eqref{eq:def_H}, and $q_H$ the associated quotient map
to $G = \Rddd/H$.
Since we are in the abelian setting and $H$ is closed by Lemma \ref{lem:H_properties}, $H$ is a Lie group with Lie algebra $\mathfrak{h}$ \cite[p.~51.]{Knapp_LieBeyond}. 
Furthermore, 
$G$ is a connected abelian Lie group with Lie algebra $\mathfrak{g}=\Rddd/\mathfrak{h}\cong \R^k$ for some $k\leq 4d$.

\emph{An injective Lie group homomorphism.} 
By Lemma \ref{lem:Phi_of_Sum}, for all $\lambda,\mu\in\Rddd$
    \begin{equation}
        \Phi(\lambda+\mu) -\Phi(\lambda)-\Phi(\mu)\in H.
    \end{equation}
    Thus,
    \begin{equation}
        (q_H\circ \Phi)(\lambda+\mu) = (q_H\circ \Phi)(\lambda)+(q_H\circ \Phi)(\mu),
    \end{equation}
    that is, $\pi = q_H\circ \Phi$ is a group homomorphism. By Lemma \ref{lem:Phi_almost_uniq}, it is one-to-one.

    Since $q_H$ is continuous and $\Phi$ is measurable, $\pi$ is
    measurable. By \cite[Thm.~22.18]{HewittRoss1979}, this implies
    that $\pi$ is continuous. The continuity of $\pi $ implies that it
    is even 
    real analytic \cite[Chap.~4 \textsection~XIII,~Prop.~1]{Chevalley1946}, 
    see also \cite[Chap.~4 \textsection~I]{Chevalley1946}.
    In particular, 
    $\pi $ is a Lie group homomorphism from $\Rddd$ to $G = \Rddd/H$. It induces a \emph{Lie algebra homomorphism} between the corresponding Lie algebras $\Rddd$ and $\mathfrak{g} =\R^{k}$. 

\emph{Dimension count.} 
Since $\pi$ is one-to-one, so is $d\pi$ \cite[Cor.~2.7.4.]{Varadarajan1984}. Thus, $4d\leq k$, i.e., $k=4d$. 
This implies that $d\pi$ is actually invertible and $\mathfrak{h}=\{0\}$.
The latter implies that $H$ is \emph{discrete}, i.e., $H\cong \Z^l$ for some $0\leq l\leq 4d$.
    
    \emph{Inverting $\pi$.}
    Since $d\pi$ is onto, the image $\pi(\Rddd)$ is an open connected subgroup of $G$ \cite[Cor.~2.7.4.]{Varadarajan1984}.
    However, $G$ is connected, so the only open connected subgroup of $G$ is $G$ itself. 
    Thus $\pi$ is invertible, with a real analytic inverse. Therefore, $\pi$ is a Lie group isomorphism.
    This implies that the simply connected Lie group $\Rddd$ is topologically isomorphic to $G\cong \R^{4d-l}\times \T^l$. 
    But the torus is not simply connected, hence $l=0$, i.e., $H=\{0\}$.
    This also implies that $\pi =\Phi$ is an automorphism on $\Rddd$, i.e., there exists an invertible matrix 
    $\A\in\mathrm{GL}(4d,\R)$ such that $\Phi(\lambda) = \A \lambda$, $\lambda\in\Rddd$.
This concludes the proof.
\end{proof}

It remains to prove the main result.

\begin{proof}[Proof of Theorem \ref{thm:main}]
    By Lemma \ref{lem:Trep_extension}, there exists a weak*-continuous operator $\Trep:\Sdd\to\Sddp$ and a measurable function $\Phi:\Rddd\to\Rddd$ such that for all $\lambda\in\Rddd$ and all $F\in\Sdd$ holds
\begin{equation}
    \Trep \rho(\lambda)F =  c \,  \rho(\Phi(\lambda)) \Trep F.
\end{equation}
By Proposition \ref{prop:Phi_linear}, there exists an invertible matrix $\A\in\mathrm{GL}(4d,\R)$ such that $\Phi(\lambda) = \A \lambda$, $\lambda\in\Rddd$.
By Lemma \ref{eq:Phi_of_Sum}, $\A$ satisfies
\begin{equation}
    [\lambda,\mu]-[\A\lambda,\A\mu]\in\Z,\qquad \lambda,\mu\in\Rddd.
\end{equation}
Since the expression is continuous in $\lambda$ and $\mu$, this implies that 
\begin{equation}
    [\lambda,\mu]-[\A\lambda,\A\mu]=0,\qquad \lambda,\mu\in\Rddd.
\end{equation}
i.e., $\A$ is symplectic.

Define $\Trep _0: \Sdd \to \Sddp$ as 
$$
\Trep_0\coloneqq \hat\A^{-1}\Trep \, .
$$
Then by  \eqref{eq:Stone_von-Neumann}, $\Trep _0$  satisfies the
identity 
\begin{equation}
    \Trep_0 \rho(\lambda)F = \hat\A^{-1}\Trep \rho(\lambda)  F =  \hat\A^{-1}\rho(\A\lambda) \Trep F =   \rho(\lambda) \hat\A^{-1} \Trep F =   \rho(\lambda) \Trep_0 F,
\end{equation}
that is, $\Trep _0$  commutes with all \tf\ shifts. 
This fact implies that $\Trep_0$ is a multiple of the identity
operator on $\Sddp$. 

If $\Trep _0$ were bounded on $\lrd $, then this conclusion would follow  from
Schur's lemma.  On the level of distributions, one may  argue as follows.
Since $\Trep_0$ commutes in particular with all time shifts $T_x=\rho(x,0)$, it is a convolution operator \cite[p.~169]{GelfandVilenkin1964}, i.e., there exists a $K_0\in\Sddp$ such that
\begin{equation}
    \Trep_0 F = K_0*F,\quad F\in\Sdd
\end{equation}
holds in the distributional sense. 
Since $\Trep_0$ also commutes with modulations, 
\begin{equation} 
   K_0 *F = M_{-\omega} M_\omega \Trep_0 F 
   = M_{-\omega} \Trep_0 M_\omega F  
   = M_{-\omega} (K_0 * M_\omega F).
\end{equation}
Since $\langle M_\omega K_0,G\rangle  = \langle K_0, M_\omega G\rangle $ and 
$\langle K_0*F,G\rangle = \langle K_0, \tilde F*G\rangle$, where $\tilde F(t) = F(-t)$ denotes the reflection of $F$, 
the last equation can be rewritten as 
\begin{align*}
    \langle K_0, \tilde F* G\rangle  
    & = \langle M_{\omega} (K_0 * M_{-\omega} F), G\rangle  
     = \langle K_0 * M_{-\omega} F, M_{\omega}G \rangle \\
    & = \langle K_0 ,  (M_{-\omega} F)\, \tilde{}* (M_{\omega}G) \rangle
     = \langle K_0 , (M_{\omega} \tilde F)* (M_{\omega}G) \rangle \\
    & = \langle K_0 , M_{\omega} (\tilde F* G) \rangle 
     = \langle M_{\omega}  K_0 , \tilde F* G \rangle.    
    \end{align*}
    Now,  $\Sdd*\Sdd$ spans a dense subspace of $\Sdd$, because it
    contains all Gaussian functions. Therefore it follows that $K_0 = M_\omega K_0$ for all $\omega\in\Rdd$. 
    This implies that $K_0$ is supported on $\{ 0\}$. The only distribution supported on $\{0\}$ which commutes with all frequency shifts is a multiple of the identity, i.e., there exists an $a\in\C$ such that
\begin{equation}
    \hat\A^{-1} \Trep F = \Trep_0 F = aF,\qquad F\in\Sdd.
\end{equation}
Equivalently, $\Trep = a\hat\A$ on $\Sd$. This concludes the proof of
the main theorem.
\end{proof}

It seems that a homomorphism from a locally compact group $G$ to a
quotient $G/H$ with respect to a closed normal  subgroup $H$ can be
one-to-one only if $H=\{e\}$, but currently we do not have an argument
for this more general result. 

\section{Final Remarks}

To appreciate the proof above, we present  some alternative arguments. They do not
require assume the measurability of $\Phi $, but require additional assumptions on the
bilinear form $\rep $.  We note that in  our formulation of
Theorem~\ref{thm:main} these  assumptions appear as corollaries. 

\emph{Eigendistributions.}  The following argument is  based on easy facts about the
eigen-distributions of \tfs s. It is easy to see that an
eigen-distribution  $\rho (\lambda ) F = F$  must be in $F\in \Sddp
\setminus \bigcup _{p\leq \infty } L^p(\rdd )$. Just write $\lambda =
\A (x,0)$ for $\A \in \Spp $ and  $\rho (\lambda ) F  = \hat{\A} \rho
(x,0) \hat{\A }\inv F = F$, then $\hat{\A }\inv F$ is periodic, and
periodic distributions cannot be in $L^p(\rdd )$ for $p< \infty $.
Using  this observation, we sketch a different version  of
Proposition~\ref{prop:Phi_linear} under additional hypothesis. 

\begin{proposition} 
  Let $\Trep:\Sdd \to \Sddp$ be a non-zero linear operator
    satisfying  the covariance property $        \Trep \rho(\lambda)F
    = c \rho(\Phi(\lambda)) \Trep F$ for all $ \lambda\in\Rddd $ and
    all $F\in\Sdd$.

    Assume in addition that

    (i) either there exist $ F\in\Sdd $ with a non-zero
    image  $\Trep F
    \in L^p(\rdd )$ for some $p< \infty $,   or

    (ii) $\{\Trep F :  F\in\Sdd \}$ is a
    weak$^*$-dense subspace of $\Sddp $.

    Then $\Phi $ is an additive map. 
 \end{proposition}

 Note that we do not assume that $\phi$ is measurable.  Assumption (i) is quite reasonable, because a \tfr\ of
Schwartz functions should be ``nice''. Assumption (ii) is very strong,
but difficult to verify. It  actually follows as a corollary from Theorem~\ref{thm:main}
provided that we assume the measurability of $\phi$. 

\begin{proof}[Sketch of proof]
  By Lemma \ref{lem:Phi_of_Sum}
  $$
  \rho(\Phi(\lambda_1+\lambda_2)-\Phi(\lambda_1)-\Phi(\lambda_2))
  \Trep F = c \Trep F
  $$
  with some phase factor $c$. Thus $\Trep F \neq 0 $ is an eigen-distribution
  of a \tfs . For  $\Trep F \in L^p(\rdd )$ this is only possible, if
  the \tfs\ is a multiple of the  identity, whence
  $\Phi(\lambda_1+\lambda_2)-\Phi(\lambda_1)-\Phi(\lambda_2)=0$  and
  thus    $\Phi$ is additive.

  If $\Trep (\Sd\otimes \Sd)$ spans a weak$^*$-dense subspace of
  $\Sddp$, then again the \tfs\ on the left-hand side must be a
  multiple of the identity, and $\Phi $  is additive.  
\end{proof}

\emph{Uniqueness.} Some arguments would simplify, if we knew that
the distortion function $\Phi $ were unique. 

\begin{proposition}\label{thm:main_uniqueV} 
    Let $\rep:\Sd\times \Sd\to \Sddp$ be a non-zero, bilinear,
    separately weak*-continuous mapping satisfying the general
    covariance property ~\eqref{eq:vertauschung}.
    Further, assume that $\Phi_\rep $ is unique. 
    Then there exist $a\in\C$, $a\neq 0$, and $\A\in\Sp$ such that
    \begin{equation} \label{metaform}
        \rep(f,g) = a\, \hat\A (f\otimes g),\qquad f,g\in\Sd.
    \end{equation}
  \end{proposition}

\begin{proof}[Proof sketch]
The only part where measurability plays a role in the proof above is
related to Proposition \ref{prop:Phi_linear}. We indicate how to use
uniqueness instead. 
The rest of the proof is the same as the proof of Theorem
\ref{thm:main}. A simple argument, which we omit, shows that the
constant in~\eqref{eq:vertauschung} can be taken to be $c=1$.

Now for $\lambda \in \Rddd $, $n\in \bN $, and $F\in \Sdd $  we have
\begin{align*}
   \rho (\Phi (n\lambda )) \Trep F & =  \Trep \rho(n\lambda)F = \rho
                                     (\Phi (\lambda )) \Trep \phi
                                     \big((n-1)\lambda \big)F   =
                                     \dots \\
    & = \rho (\Phi (\lambda ))^n \Trep F = \rho \big( n (\Phi
      (\lambda )\big) \Trep F.
\end{align*}
By uniqueness of $\Phi $ we obtain
$$
\Phi (n\lambda ) = n \Phi (\lambda ) \, ,
$$
which implies even that $\Phi $ is $\bQ $-homogeneous.

We now  apply  \eqref{eq:Z_preserved}, where the measurability of $\Phi $ is
not used,   and obtain for all $m,n\in\Z$
\begin{equation}
    2^{m+n}([\lambda,\mu]-[\Phi(\lambda), \Phi(\mu)]) = [2^m\lambda,2^n\mu]-[\Phi(2^n\lambda), \Phi(2^n \mu)]\in\Z.
\end{equation}
This is equivalent to 
\begin{equation}
    [\lambda,\mu]-[\Phi(\lambda), \Phi(\mu)] = 0,
\end{equation}
which already implies that $\Phi$ is a linear symplectic map. 

The rest of the proof is identical to the Proof of Theorem \ref{thm:main}.
\end{proof}

\end{document}